\documentclass[11pt]{article}

\usepackage{amssymb,amsthm,amsmath}

\def\Ex{{\mathbb E}}
\def\Pr{{\mathbb P}}
\def\er{{\mathbb R}}
\def\ind{{\mathbf 1}}

\newtheorem{lem}{Lemma}

\newtheorem{thm}[lem]{Theorem}
\newtheorem{prop}[lem]{Proposition}
\newtheorem{cor}[lem]{Corollary}

\title{Modified Paouris inequality}
\author{Rafa{\l} Lata{\l}a\thanks{Research supported by the NCN grant DEC-2012/05/B/ST1/00412.}}
\date{}

\begin{document}

\maketitle

\begin{abstract}
The Paouris inequality gives the large deviation estimate for Euclidean norms of log-concave vectors. 
We present a modified version of it and show how the new inequality may be applied to derive tail estimates of 
$l_r$-norms and suprema of norms of coordinate projections of isotropic log-concave vectors.
\end{abstract}

\section{Introduction and Main Results}

A random vector $X$ is called \emph{log-concave} if it has a logarithmically concave distribution, i.e.
$\Pr(X\in \lambda K+(1-\lambda)L)\geq \Pr(X\in K)^{\lambda}\Pr(X\in L)^{1-\lambda}$
for all nonempty compact sets $K,L$ and $\lambda\in [0,1]$. 
The result of Borell \cite{Bo} states that a random vector with the
full dimensional support is log-concave iff it has a logconcave density, i.e. a density of the form $e^{-h(x)}$, where
$h$ is a convex function with values in $(-\infty,\infty]$. A typical example of a log-concave vector is a vector uniformly
distributed over a convex body. In recent years the study of log-concave vectors attracted attention of many researchers,
cf. the forthcoming monograph \cite{BGVV}.

The fundamental result of Paouris \cite{Pa} gives the large deviation estimate for Euclidean norms of log-concave vectors.
It may be stated, c.f.\ \cite{ALLPT}, in the form 
\[
(\Ex|X|^p)^{1/p}\leq C_1(\Ex|X|+\sigma_X(p)) \quad \mbox{ for any }p\geq 1, 
\] 
and any log-concave vector $X$, where here and in the sequel $C_i$ denote universals constant, $|x|$ is the canonical 
Euclidean norm on $\er^n$ and
\[
\sigma_X(p):=\sup_{|t|=1}(\Ex|\langle t,X\rangle|^p)^{1/p},\quad p\geq 1.
\]
In particular if $X$ is additionally \emph{isotropic}, i.e. it is centered and have identity covariance matrix then
\begin{equation}
\label{eq:Paour}
(\Ex|X|^p)^{1/p}\leq C_1(\sqrt{n}+\sigma_X(p)) \quad \mbox{ for }p\geq 1. 
\end{equation}

In this note we show the following modification of the Paouris inequality.

\begin{thm}
\label{thm:cutPaour}
For any isotropic log-concave $n$-dimensional random vector $X$ and $p\geq 1$,
\begin{equation}
\label{eq:cutPaour}
\Ex\Big(\sum_{i=1}^nX_i^2\ind_{\{|X_i|\geq t\}}\Big)^p\leq (C_2\sigma_X(p))^{2p} \quad \mbox{ for } 
t\geq C_2\log\Big(\frac{n}{\sigma_X(p)^2}\Big). 
\end{equation}
\end{thm}

Obviuosly $\sum_{i=1}^nX_i^2\ind_{\{|X_i|\geq t\}}\geq t^2N_X(t)$, where
\[
N_X(t):=\sum_{i=1}^n\ind_{\{|X_i|\geq t\}},\quad t>0,
\]
thus \eqref{eq:cutPaour} generalizes the estimate derived in \cite{ALLPT}:
\[
\Ex(t^2N_X(t))^p\leq (C\sigma_X(p))^{2p} \quad \mbox{ for } t\geq C\log\Big(\frac{n}{\sigma_X(p)^2}\Big).  
\]

It is also not hard to see that Theorem \ref{thm:cutPaour} implies Paouris inequality \eqref{eq:Paour}. To see this
let $p':=\inf\{q\geq p\colon\ \sigma_X(q)\geq \sqrt{n}\}$. Then 
\[
(\Ex|X|^p)^{1/p}\leq (\Ex|X|^{2p'})^{1/2p'}\leq C_2\sigma_X(p')\leq C_2(\sqrt{n}+\sigma_X(p)),
\]
where the second inequality follows by \eqref{eq:cutPaour} aplied with $p=p'$ and $t=0$. 

In fact we may extend estimate \eqref{eq:Paour} replacing the Euclidean norm by the $l_r$-norm, 
$\|x\|_r:=(\sum_{i}|x_i|^r)^{1/r}$, $r\geq 2$.

\begin{thm}
\label{thm:lrPaour}
For any $r\geq 2$ and any isotropic log-concave $n$-dimensional random vector $X$, 
\begin{equation}
\label{eq:lrPaour}
(\Ex\|X\|_r^p)^{1/p}\leq C_3(rn^{1/r}+\sigma_X(p)) \quad \mbox{ for }p\geq 1.
\end{equation}
\end{thm}

Theorem \ref{thm:lrPaour} gives better bounds than presented in \cite{La}, since the constant
does not explode for $r\rightarrow 2+$ and the parameter $p$ is replaced by the smaller quantity $\sigma_X(p)$.
Estimate \eqref{eq:lrPaour} and Chebyshev's inequality imply for $t\geq 1$,
\[
\Pr(\|X\|_r\geq 2eC_3trn^{1/r})\leq \exp(-\sigma_{X}^{-1}(trn^{1/r})).
\]

In general \eqref{eq:lrPaour} is sharp up to a multiplicative constant, since for a random vector $X$ with 
i.i.d. symmetric exponential coordinates with variance $1$ we have $\sigma_X(p)\leq p\sigma_X(2)=p$
and
\[
(\Ex\|X\|_r^p)^{1/p} \geq \max\{\Ex\|X\|_r,(\Ex|X_1|^p)^{1/p})\geq \frac{1}{C}\max\{rn^{1/r},p\}.
\]
However there are reasons to believe that the following stronger estimate may hold for log-concave vectors 
(c.f. \cite{La2}) 
\[
(\Ex\|X\|_r^p)^{1/p} \leq C\Big(\Ex\|X\|_r+\sup_{\|t\|_{r'}\leq 1}(\Ex|\langle t,X\rangle|^p)^{1/p}\Big).
\]

\medskip

Another consequence of Theorem \ref{eq:cutPaour} is the uniform version of the Paouris inequality.
For $I\subset\{1,\ldots,n\}$ by $P_I$ we denote the coordinate projection from $\er^n$ into $\er^I$.

\begin{thm}
\label{thm:unifP}
For any isotropic log-concave $n$-dimensional random vector $X$ and $1\leq m\leq n$ we have
\begin{equation}
\label{eq:unifP}
\Big(\Ex\max_{|I|=m}|P_IX|^p\Big)^{1/p}\leq C_4\Big(\sqrt{m}\log\Big(\frac{en}{m}\Big)+\sigma_X(p)\Big) 
\quad \mbox{ for }p\geq 1.
\end{equation}
\end{thm}

Again the example of a vector with the product isotropic exponential distribution shows that 
in general estimate \eqref{eq:unifP}
is sharp. Theorem \ref{thm:unifP} and Chebyshev's inequality yield for $t\geq 1$,
\[
\Pr\Big(\max_{|I|=m}|P_IX|\geq 2eC_4t\sqrt{m}\log\Big(\frac{en}{m}\Big)\Big)\leq
\exp\Big(-\sigma_X^{-1}\Big(t\sqrt{m}\log\Big(\frac{en}{m}\Big)\Big)\Big),
\]
which removes an exponential factor from Theorem 3.4 in \cite{ALLPT}.

\medskip

The paper is organised as follows. In Section 2 we recall basic facts about log-concave vectors and prove Theorem
\ref{thm:cutPaour}. In Section 3 we show how to use \eqref{eq:cutPaour} to get estimates for the joint distribution of
order statistics of $X$ and derive Theorems \ref{thm:lrPaour} and \ref{thm:unifP}.    

{\bf Notation.} For a r.v. $Y$ and $p>0$ we set $\|Y\|_p:=(\Ex|Y|^p)^{1/p}$. We write $|I|$ for 
the cardinality of a set $I$. By a letter $C$ we denote
absolute constants, value of $C$ may differ at each occurence. Whenever we want to fix a value of an absolute
constant we use letters $C_1,C_2,\ldots$.

\section{Proof of Theorem \ref{thm:cutPaour}}

The result of Barlow, Marshall and Proschan \cite{BMP} imply that for symmetric log-concave random variables $Y$,
and $p\geq q>0$, $\|Y\|_p\leq \Gamma(p+1)^{1/p}/\Gamma(q+1)^{1/q}\|Y\|_q$.
If $Y$ is centered and log-concave and $Y'$ is
an independent copy of $Y$ then $Y-Y'$ is symmetric and log-concave, hence for $p\geq q\geq 2$,
\[
\|Y\|_p\leq \|Y-Y'\|_p\leq \frac{\Gamma(p+1)^{1/p}}{\Gamma(q+1)^{1/q}}\|Y-Y'\|_q\leq
2\frac{\Gamma(p+1)^{1/p}}{\Gamma(q+1)^{1/q}}\|Y\|_q\leq 2\frac{p}{q}\|Y\|_q.
\]
Thus for isotropic log-concave vectors $X$,
\[
\sigma_X(\lambda p)\leq 2\lambda\sigma_X(p)\quad \mbox{ and }\quad 
\sigma_X^{-1}(\lambda t)\geq \frac{\lambda}{2}\sigma_X^{-1}(t)
\quad \mbox{ for }p\geq 2,\ t,\lambda\geq 1.
\]
In particular $\sigma_X(p)\leq p$ for $p\geq 2$.

If $Y$ is a log-concave r.v.\ (not necessarily centered) then for $p\geq 2$, 
$\|Y\|_p\leq |\Ex Y|+\|Y-Y'\|_p\leq (p+1)\|Y\|_2$ and Chebyshev's inequality yields
$\Pr(|Y|\geq e(p+1)\|Y\|_2)\leq e^{-p}$. Thus we obtain a $\Psi_1$-estimate for log-concave r.v's
\begin{equation}
\label{eq:psi1}
\Pr(|Y|\geq t)\leq \exp\Big(2-\frac{t}{2e\|Y\|_2}\Big) \quad \mbox{ for }t\geq 0.
\end{equation}

We start with a variant of Proposition 7.1 from \cite{ALLPT}.

\begin{prop}
\label{conditional} 
There exists an absolute positive constant $C_5$  such that the following holds. Let $X$ be an isotropic log-concave
$n$-dimensional random vector, $A=\{X\in K\}$, where $K$ is a convex set in $\er^n$ satisfying $0<\Pr(A)\leq 1/e$. 
Then for every $t\geq 1$,  
\begin{equation}
\label{eq:cond1}
\sum_{i=1}^n\Ex X_i^2\ind_{A\cap\{X_i\geq t\}}\leq 
C_5\Pr(A)\Big(\sigma_X^2(-\log(\Pr(A)))+nt^2e^{-t/C_5}\Big) 
\end{equation}
and for every $t>0$, $u\geq 1$, 
\begin{align}
\notag
\sum_{k=0}^{\infty}4^k|\{i\leq n\colon\ \Pr(A\cap\{X_i\geq &2^k t\})\geq e^{-u}\Pr(A)\}|
\\
\label{eq:cond2}
&\leq
\frac{C_5 u^2}{t^2}\Big(\sigma_X^2(-\log(\Pr(A)))+n\ind_{\{t\leq uC_5\}}\Big) .
\end{align}
\end{prop}

\begin{proof}
Let $Y$ be a random vector defined by 
\[
\Pr(Y\in B)=\frac{\Pr(A\cap\{X\in B\})}{\Pr(A)}=\frac{\Pr(X\in B\cap K)}{\Pr(X\in K)},  
\]
i.e. $Y$ is distributed as $X$ conditioned on $A$. Clearly, for every measurable set $B$  
one has $\Pr(X\in B)\geq \Pr(A)\Pr(Y\in B)$. It is easy to see that $Y$ is log-concave, but not necessarily isotropic. 

The Paouris inequality \eqref{eq:Paour} (applied for a vector $P_IX$) implies that for any $\emptyset\neq I\subset \{1,\ldots,n\}$ and $t\geq (2eC_1)^2|I|$,
\begin{equation}
\label{eq:Pa}
\Pr\Big(\sum_{i\in I}X_i^2\geq t\Big)=\Pr(|P_IX|\geq \sqrt{t})\leq \exp\Big(-\sigma_X^{-1}\Big(\frac{1}{2eC_1}\sqrt{t}\Big)\Big).
\end{equation}

Let 
\[
I:=\{i\leq n\colon\ \Ex Y_i^2\geq 2(2eC_1)^2\}. 
\]
Using the Paley-Zygmund inequality and log-concavity of $Y$, we get 
\[
\Pr\Big(\sum_{i\in I} Y_i^2\geq \frac{1}{2}\sum_{i\in I}\Ex Y_i^2\Big)
\geq
\frac{1}{4}\frac{(\Ex \sum_{i\in I} Y_i^2)^2}{\Ex(\sum_{i\in I} Y_i^2)^2}\geq  
\frac{1}{C_6}.
\]
Therefore 
\[
\Pr\Big(\sum_{i\in I} X_i^2\geq \frac{1}{2}\sum_{i\in I}\Ex Y_i^2\Big)\geq 
\Pr(A)\,  \Pr\Big(\sum_{i\in I} Y_i^2\geq \frac{1}{2}\sum_{i\in I}\Ex Y_i^2\Big)\geq 
\frac{1}{C_6}\Pr(A).
\]
Together with \eqref{eq:Pa} this gives 
\[
\frac{1}{C_6}\Pr(A)\leq \exp\Big(-\sigma_X^{-1}\Big(\frac{1}{2eC_1}\sqrt{\frac{1}{2}\sum_{i\in I}\Ex Y_i^2}\Big)\Big),
\]
hence
\[
\sum_{i\in I}\Ex Y_i^2\leq C\sigma_X^2(-\log\Pr(A)).
\]
Moreover if $i\notin I$, i.e. $\Ex Y_i^2\leq 2(2eC_1)^2$ then \eqref{eq:psi1} yields 
$\Ex Y_i^2\ind_{\{|Y_i|\geq t\}}\leq Ct^2e^{-t/C}$ for $t\geq 1$. Therefore
\begin{align*}
\sum_{i=1}^n\Ex X_i^2\ind_{A\cap\{|X_i|\geq t\}}&=\Pr(A)\sum_{i=1}^n\Ex Y_i^2\ind_{\{|Y_i|\geq t\}}\leq
\Pr(A)\Big(\sum_{i\in I}\Ex Y_i^2+nCt^2e^{-t/C})
\\
&\leq C\Pr(A)\Big(\sigma_X^2(-\log(\Pr(A)))+nt^2e^{-t/C}\Big). 
\end{align*}

To show \eqref{eq:cond2} note first that for every $i$ the random variable $Y_i$ is log-concave, hence for $s\geq 0$,
\[
\frac{\Pr(A\cap\{X_i\geq s\})}{\Pr(A)}=
\Pr(Y_i\geq s)\leq \exp\Big(2-\frac{t}{2e\|Y_i\|_2}\Big).
\]
Thus, if $\Pr(A\cap\{X_i\geq 2^kt\})\geq e^{-u}\Pr(A)$ and $u\geq 1$ then $\|Y_i\|_2\geq
2^kt/(2e (u+2))\geq 2^kt/(6eu)$.  In particular it cannot happen if  $i\notin I$, $k\geq 0$ and $u\leq t/C_5$ with 
$C_5$ large enough. 

Therefore
\begin{align*}
\sum_{k=0}^{\infty}4^k|\{i\leq n\colon\ &\Pr(A\cap\{X_i\geq 2^k t\})\geq e^{-u}\Pr(A)\}|
\\
&\leq
\Big(\sum_{i\in I}+\ind_{\{t\leq uC_5\}}\sum_{i\notin I}\Big)\sum_{k=0}^{\infty}4^k\ind_{\{(\Ex Y_i^2)^{1/2}\geq 2^kt/(6eu)\}}
\\
&\leq \frac{2(6eu)^2}{t^2}\Big(\sum_{i\in I}+\ind_{\{t\leq uC_5\}}\sum_{i\notin I}\Big)\Ex Y_i^2
\\
&\leq \frac{Cu^2}{t^2}\Big(\sigma_X^2(-\log(\Pr(A)))+n\ind_{\{t\leq uC_5\}}\Big).
\end{align*}

\end{proof}

We will also use the following simple combinatorial lemma (Lemma~11 in \cite{La}).

\begin{lem}
\label{combf}
Let $\ell_0\geq \ell_1\geq\ldots\geq \ell_s$ be a fixed sequence of
positive integers and
\[
{\cal  F}:=\Big\{f\colon\{1,2,\ldots,l_0\} \rightarrow\{0,1,2,\ldots,s\}\colon\ 
\forall_{1\leq i\leq s}\ |\{r\colon f(r)\geq i\}|\leq l_i\Big\}.
\]
Then
\[
|{\cal F}|\leq\prod_{i=1}^s \Big(\frac{e l_{i-1}}{l_i}\Big)^{l_i}.
\]
\end{lem}

\begin{proof}[Proof of Theorem \ref{thm:cutPaour}]

We have by the Paouris estimate \eqref{eq:Paour}
\[
\Ex\Big(\sum_{i=1}^nX_i^2\ind_{\{|X_i|\geq t\}}\Big)^p\leq \Ex|X|^{2p}\leq 
(C_1(\sqrt{n}+\sigma_X(2p)))^{2p},
\]
so the estimate \eqref{eq:cutPaour} is obvious if $\sigma_X(p)\geq \frac{1}{8}\sqrt{n}$, we will thus assume that
$\sigma_X(p)\leq \frac{1}{8}\sqrt{n}$.

Observe that for $l=1,2,\ldots$,
\begin{align*}
\Ex\Big(\sum_{i=1}^nX_i^2\ind_{\{X_i\geq t\}}\Big)^l&\leq 
\Ex\Big(\sum_{i=1}^n\sum_{k=0}^{\infty}4^{k+1}t^2\ind_{\{X_i\geq 2^kt\}}\Big)^l
\\
&=(2t)^{2l}\sum_{i_1,\ldots,i_l=1}^n\sum_{k_1,\ldots,k_l=0}^{\infty} 4^{k_1+\ldots+k_l}\Pr(B_{i_1,k_1\ldots,i_l,k_l}),
\end{align*}
where 
\[
B_{i_1,k_1\ldots,i_l,k_l}:=\{X_{i_1}\geq 2^{k_1}t,\ldots,X_{i_l}\geq 2^{k_l}t_l\}.
\]

Define a positive integer $l$ by 
\[
p< l\leq 2p \quad \mbox{ and }\quad l=2^m \mbox{ for some positive integer }m. 
\]
Then $\sigma_X(p) \leq \sigma_X(l)\leq \sigma_X(2p)\leq 4\sigma_X(p)$. 
Since $-X$ is also isotropic log-concave and  for any nonnegative r.v.\ $Y$, $(\Ex Y^p)^{1/p}\leq (\Ex Y^l)^{1/l}$,  
it is enough to show that
\begin{equation}
\label{eq:aim}
m(l):=\sum_{k_1,\ldots,k_l=0}^{\infty}\sum_{i_1,\ldots,i_l=1}^n 4^{k_1+\ldots+k_l}\Pr(B_{i_1,k_1\ldots,i_l,k_l})\leq
\Big(\frac{C\sigma_X(l)}{t}\Big)^{2l} 
\end{equation}
provided that $t\geq C_2\log(\frac{n}{\sigma_X(l)^2})$. Since $\sigma_X(l)\leq 4\sigma_X(p)\leq \frac{1}{2}\sqrt{n}$
this in particular implies that $t\geq C_2$.  

We divide the sum in $m(l)$ into several parts.
Define sets 
\[
I_{0}:=\big\{(i_1,k_1,,\ldots,i_l,k_l)\colon\   \Pr(B_{i_1,k_1,\ldots,i_l,k_l})> e^{-l}\big\},
\]
and for $j=1,2,\ldots$,
\[
 I_{j}:=\big\{(i_1,k_1,,\ldots,i_l,k_l)\colon\   \Pr(B_{i_1,k_1,\ldots,i_l,k_l})\in (e^{-2^{j}l},e^{-2^{j-1}l}] \big\}.
\]
Then $m(l)=\sum_{j\geq 0}m_j(l)$, where
\[
m_{j}(l):= \sum_{(i_1,k_1,\ldots,i_l,k_l)\in I_j} 4^{k_1+\ldots+k_l}\Pr(B_{i_1,k_1\ldots,i_l,k_l}).
\]

To estimate $m_{0}(l)$ define for $1\leq s\leq l$,
\[
P_sI_0:=\{(i_1,k_1,\ldots,i_s,k_s)\colon (i_1,k_1,\ldots,i_l,k_l)\in I_0 \mbox{ for some }
i_{s+1},\ldots,k_l\}.
\]
We have by \eqref{eq:psi1} (if $C_2$ is large enough)
\[
\Pr(B_{i_1,k_1\ldots,i_s,k_s})\leq \Pr(B_{i_1,k_1})\leq \exp(2-2^{k_1-1}t/e)\leq e^{-1}.
\]
Thus for $s=1,\ldots,l-1$,  
\begin{align*}
&t^2\sum_{(i_1,\ldots,k_{s+1})\in P_{s+1}I_{0}}4^{k_1+\ldots+k_{s+1}}\Pr(B_{i_1,\ldots,k_{s+1}})
\\
&\leq
\sum_{(i_1,\ldots,k_s)\in P_{s}I_{0}}4^{k_1+\ldots+k_{s}}\sum_{i_{s+1}=1}^n\sum_{k_{s+1}=0}^{\infty}
4^{k_{s+1}}t^2\Pr(B_{i_1,\ldots,k_s}\cap\{X_{i_{s+1}}\geq 2^{k_{s+1}}t\})
\\
&\leq
\sum_{(i_1,\ldots,k_s)\in P_{s}I_{0}}4^{k_1+\ldots+k_{s}}
\sum_{i_{s+1}=1}^n\Ex 2X_{i_{s+1}}^2\ind_{B_{i_1,\ldots,k_s}\cap\{X_{i_{s+1}}\geq t\}}
\\
&\leq
2C_5\sum_{(i_1,\ldots,k_s)\in P_{s}I_{0}}4^{k_1+\ldots+k_{s}}
\Pr(B_{i_1,\ldots,k_s})(\sigma_X^2(-\log\Pr(B_{i_1,\ldots,k_s}))+nt^2e^{-t/C_5}),
\end{align*}
where the last inequality follows by \eqref{eq:cond1}.
Note that for $(i_1,\ldots,k_s)\in P_sI_0$ we have
$\Pr(B_{i_1,\ldots,k_s})\geq e^{-l}$ and, by our assumptions on $t$ (if $C_2$ is sufficiently large)
$nt^2e^{-t/C_5}\leq ne^{-t/(2C_5)} \leq \sigma_X^2(l)$.  Therefore
\begin{align*}
\sum_{(i_1,\ldots,k_{s+1})\in P_{s+1}I_{0}}&4^{k_1+\ldots+k_{s+1}}\Pr(B_{i_1,\ldots,k_{s+1}})
\\
&\leq 4C_5t^{-2}\sigma_X^2(l)
\sum_{(i_1,\ldots,k_{s})\in P_{s}I_{0}}4^{k_1+\ldots+k_{s}}\Pr(B_{i_1,\ldots,k_{s}}).
\end{align*}
By induction we get
\begin{align*}
m_{0}(l)&
=\sum_{(i_1,\ldots,k_l)\in I_{0}}4^{k_1+\ldots+k_{l}}\Pr(B_{i_1,\ldots,k_l})
\\
&\leq
(4C_5t^{-2}\sigma_X^2(l))^{l-1}
\sum_{(i_1,k_1)\in P_1I_0}4^{k_1}\Pr(B_{i_1,k_1})
\\
&\leq (4C_5t^{-2}\sigma_X^2(l))^{l-1}t^{-2}\sum_{i=1}^n2\Ex X_i^2\ind_{\{X_i\geq t\}}
\\
&\leq (4C_5t^{-2}\sigma_X^2(l))^{l-1}nCe^{-t/C}
\leq \Big(\frac{C\sigma_X(l)}{t}\Big)^{2l}, 
\end{align*}
where the last inequality follows from the assumptions on $t$.

Now we estimate $m_j(l)$ for $j>0$. 
Fix $j>0$ and define a positive integer $r_1$ by 
\[
2^{r_1-1}< \frac{t}{C_5}\leq 2^{r_1}.
\]
For all $(i_1,k_1,\ldots,i_l,k_l)\in I_j$ define a function 
$f_{i_1,k_1,\ldots,i_l,k_l}\colon \{1,\ldots,\ell\}\rightarrow \{0,1,2,\ldots\}$ 
by 
\[
f_{i_1,k_1,\ldots,i_l,k_l}(s):=
\left\{
\begin{array}{ll}
0 &\mbox{if } \frac{\Pr(B_{i_1,k_1,\ldots,i_s,k_s})}{\Pr(B_{i_1,k_1,\ldots,i_{s-1},k_{s-1}})}>e^{-1}
\\
r &\mbox{if } e^{-2^{r}}< \frac{\Pr(B_{i_1,k_1,\ldots,i_s,k_s})}{\Pr(B_{i_1,k_1,\ldots,i_{s-1},k_{s-1}})}\leq e^{-2^{r-1}},
\ r\geq 1.
\end{array}
\right.
\]
Note that for every $(i_1,k_1,\ldots,i_l,k_l) \in  I_j$ one has 
\[
1= \Pr(B_{\emptyset})\geq \Pr(B_{i_1,k_1}) \geq \Pr(B_{i_1,k_1,i_2,k_2})\geq \ldots 
\geq \Pr(B_{i_1,k_1\ldots,i_l,k_l}) > \exp(-2^{j}l).
\]

Denote 
\[
{\cal F}_j:=\big\{f_{i_1,k_1,\ldots,i_l,k_l}\colon\ (i_1,k_1,\ldots,i_l,k_l)\in I_j\big\}.
\]
Then for $f=f_{i_1,k_1,\ldots,i_l,k_l}\in {\cal F}_j$ and  $r\geq 1$ 
one has
\begin{align*}
\exp(-2^{j}l)< \Pr(B_{i_1,k_1,\ldots,i_l,k_l}) 
&= \prod_{s=1}^{\ell}  \frac{\Pr(B_{i_1,k_1\ldots,i_s,k_s})}{\Pr(B_{i_1,k_1,\ldots,i_{s-1},k_{s-1}})}  
\\
&\leq \exp(-2^{r-1}|\{s\colon\ f(s)\geq r\}|).
\end{align*}
Hence for every $r\geq 1$ one has
\begin{equation}
\label{est_l_r}
|\{s\colon\ f(s)\geq r \}| \leq \min\{2^{j+1-r}l,l\}=:l_r.
\end{equation}
In particular $f$ takes values in $\{0,1,\ldots,j+1+\lfloor\log_2 l\rfloor\}$.
Clearly, $\sum_{r\geq 1}l_r= (j+2)l$ and $l_{r-1}/l_r\leq 2$, 
so by Lemma \ref{combf} 
\[
|{\cal F}_j| \leq \prod_{r=1}^{j+1+\lfloor\log_2 l\rfloor}\Big(\frac{e l_{r-1}}{l_r} \Big)^{l_r}\leq e^{2(j+2)l}.
\]

Now fix $f\in {\cal F}_j$ and define 
\[
I_j(f):=\{(i_1,k_1,\ldots,i_l,k_l)\colon\ f_{i_1,k_1,\ldots,i_l,k_l}=f\}
\]
and for $s\leq l$,
\[
I_{j,s}(f):=\{(i_1,k_1,\ldots,i_s,k_s)\colon\ f_{i_1,k_1,\ldots,i_l,k_l}=f\mbox{ for some }
i_{s+1},k_{s+1}\ldots,i_l,k_l\}.
\]

Recall that for $s\geq 1$, $\Pr(B_{i_1,k_1,\ldots,i_s,k_s})\leq e^{-1}$,  
moreover for $s\leq l$,
\begin{align*}
\sigma_X(-\log\Pr(B_{i_1,k_1,\ldots,i_s,k_s}))
&\leq \sigma_X(-\log\Pr(B_{i_1,k_1,\ldots,i_l,k_l}))
\leq \sigma_X(2^jl)
\\
&\leq 2^{j+1}\sigma_X(l). 
\end{align*}
Hence estimate \eqref{eq:cond2} applied with $u=2^{f(s+1)}$ implies for $1\leq s\leq l-1$,
\begin{align*}
&\sum_{(i_1,k_1,\ldots,i_{s+1},k_{s+1})\in I_{j,s+1}(f)}4^{k_1+\ldots+k_{s+1}}\Pr(B_{i_1,k_1,\ldots,i_{s+1},k_{s+1}})
\\
&\phantom{aaaaaaaaaaaaaa}
\leq g(f(s+1))\sum_{(i_1,k_1,\ldots,i_{s},k_{s})\in I_{j,s}(f)}4^{k_1+\ldots+k_{s}}\Pr(B_{i_1,k_1,\ldots,i_{s},k_{s}}),
\end{align*}
where
\[
g(r):=
\left\{
\begin{array}{ll}
C_5t^{-2}4^{r+j+1}\sigma^2_X(l)&\mbox{ if } r=1,
\\
C_5t^{-2}4^{r+j+1}\sigma^2_X(l)\exp(-2^{r-1})&\mbox{ if } 1\leq r< r_1,
\\
C_5t^{-2}4^r(4^{j+1}\sigma^2_X(l)+n)\exp(-2^{r-1})&\mbox{ if }r\geq r_1.
\end{array}
\right.
\]

Suppose that $(i_1,k_1)\in I_1(f)$ and $f(1)=r$ then
\[
\exp(-2^{r})\leq \Pr(X_{i_1}\geq 2^{k_1}t)\leq \exp(2-2^{k_1-1}t/e),
\] 
hence $2^{k_1}t\leq e2^{r+2}$. W.l.o.g.\ $C_5>4e$, therefore $r\geq r_1$.
Moreover, $4^{k_1}\leq 16e^24^rt^{-2}$, hence
\[
\sum_{(i_1,k_1)\in I_{j,1}(f)}4^{k_1}\Pr(B_{i_1,k_1})\leq n32e^2t^{-2}4^r\exp(-2^{r-1})\leq g(r)=g(f(1)),
\]
since we may assume that $C_5\geq 32e^2$.
Thus the easy induction shows that
\[
m_j(f):=\sum_{(i_1,\ldots,k_{l})\in I_{j}(f)}4^{k_1+\ldots+k_{l}}\Pr(B_{i_1,k_1,\ldots,i_{l},k_{l}})
\leq \prod_{s=1}^lg(f(s))=\prod_{r=0}^{\infty}g(r)^{n_r},
\]
where $n_r:=|f^{-1}(r)|$.

Observe that 
\[
e^{-2^{j-1}l}\geq \Pr(B_{i_1,k_1,\ldots,i_l,k_l})
=\prod_{s=1}^{l}\frac{\Pr(B_{i_1,k_1,\ldots,i_s,k_s})}{\Pr(B_{i_1,k_1,\ldots,i_{s-1},k_{s-1}})}
\geq e^{-l}\prod_{s\colon f(s)\geq 1}e^{-2^{f(s)}}
\]
therefore
\[
\sum_{r=1}^{\infty}n_r2^{r-1}=\frac{1}{2}\sum_{s\colon f(s)\geq 1}2^{f(s)}\geq 
\frac{1}{2}l(2^{j-1}-1).
\]
Moreover $4^{j+1}\sigma_X^2(l)+n\leq 2\cdot 4^{j+1}n$ and
\[
\sum_{r\geq 1} rn_r\leq (j+1)l+\sum_{r\geq j+2}rl_r= (2j+4)l.
\]
Hence
\[
\prod_{r=0}^{\infty}g(r)^{n_r}\leq 
\Big(\frac{C_54^{j+1}\sigma_X^2(l)}{t^2}\Big)^l4^{(2j+4)l}\Big(\frac{2n}{\sigma_X^2(l)}\Big)^m\exp(-\frac{l}{2}(2^{j-1}-1)).
\]
where $m=\sum_{r\geq r_1}n_r\leq l_{r_1}\leq 2^{j+1-r_1}l$.
By the assumption on $l$ we have $(2n/\sigma^2_X(l))\leq 2\exp(t/C_2)\leq \exp(2^{r_1-4})$ if $C_2$ is large enough
with respect to $C_5$.
Hence
\[
m_j(l)\leq \Big(\frac{\sqrt{e}C_54^{3j+5}\sigma_X^2(l)}{t^2}\Big)^l\exp(-l2^{j-3})
\]
and we get
\[
m(l)= \sum_{j=0}^{\infty} m_j(l)\leq  \Big(\frac{C\sigma_X(l)}{t}\Big)^{2l}+
\sum_{j\geq 1}\Big(\frac{\sqrt{e}C_54^{3j+5}\sigma_X^2(l)}{t^2}\Big)^l\exp(-l2^{j-3})
\]
and \eqref{eq:aim} easily follows. 
\end{proof}

\section{Estimates for joint distribution of order statistics}

For a random vector $X=(X_1,\ldots,X_n)$ by $X_1^*\geq X_2^*\geq\ldots\geq X_n^*$ we denote the nonincreasing
rearrangement of $|X_1|,\ldots,|X_n|$, in particular $X_1^*=\max\{|X_1|,\ldots,|X_n|\}$ and
$X_n^*=\min\{|X_1|,\ldots,|X_n|\}$. The following consequence of Theorem \ref{thm:cutPaour} generalizes Theorem 3.3
from \cite{ALLPT}.

\begin{thm}
\label{thm:jointos}
Let $X$ be an isotropic log-concave vector,
$0=l_0< l_1<l_2<\ldots<l_k\leq n$ and $t_1,\ldots,t_k\geq 0$ be such that 
\[
t_r\geq C_7\log\bigg(\frac{C_7^2n}{\sum_{j=1}^st_j^2(l_j-l_{j-1})}\bigg)\quad \mbox{ for }1\leq r\leq k. 
\]
Then
\[
\Pr\Big(X_{l_1}^*\geq t_1,\ldots,X_{l_k}^*\geq t_k\Big)
\leq \exp\Bigg(-\sigma_X^{-1}\Bigg(\frac{1}{C_7}\sqrt{\sum_{j=1}^k t_j^2(l_j-l_{j-1})}\Bigg)\Bigg).
\]
\end{thm}

\begin{proof}
Let $t:=\min\{t_1,\ldots,t_k\}$, $u:=(\sum_{j=1}^k t_j^2(l_j-l_{j-1}))^{1/2}$ and $p:=\sigma_X^{-1}(e^{-1/2}u/C_2)$.
It is not hard to see that if $C_7$ is large enough then $u\geq \sqrt{e}C_2$, so $p\geq 2$.
Assumptions imply (if $C_7$ is large enough) that $C_2\log(n/\sigma_X^2(p))=C_2\log(enC_2^2/u^2)\leq t$.
Therefore Chebyshev's inequality and Theorem \ref{thm:cutPaour} yield
\begin{align*}
\Pr\Big(X_{l_1}^*\geq t_1,\ldots,&X_{l_k}^*\geq t_k\Big)\leq 
\Pr\Big(\sum_{i=1}^nX_i^2\ind_{\{|X_i|\geq t\}}\geq u^2\Big)
\\
&\leq u^{-2p}\Ex\Big(\sum_{i=1}^nX_i^2\ind_{\{|X_i|\geq t\}}\Big)^p\leq
\Big(\frac{C_2\sigma_X(p)}{u}\Big)^{2p}\leq e^{-p}. 
\end{align*}
\end{proof}

\begin{cor}
\label{cor:jointos}
Let $X$ be an isotropic log-concave vector and 
\[
Y_j:=\Big(X_{2^{j-1}}^*-C_7\log(4n2^{-j})\Big)_{+},\quad  1\leq j\leq 1+\log_2 n. 
\]
Then for any $1\leq s\leq 1+\log_2 n$ and $u_1,\ldots,u_s\geq 0$ we have
\[
\Pr(Y_1\geq u_1,\dots,Y_s\geq u_s)
\leq  \exp\Bigg(-\sigma_X^{-1}\Bigg(\frac{1}{2C_7}\sqrt{\sum_{j=1}^s 2^{j}u_j^2}\Bigg)\Bigg).
\]
\end{cor}

\begin{proof} 
Let
\[
I=\{j\geq 0\colon\ u_j>0\}=\{i_1<\ldots<i_k\}. 
\] 
If $I=\emptyset$ there is nothing to prove, so we may assume that $k\geq 1$. Let $l_0=0$, $l_j=2^{i_j-1}$,
$t_j:=C_7\log(4n2^{-i_j})+u_{i_j}$  for $1\leq j\leq k$ and $u:=(\sum_{j=1}^k(l_j-l_{j-1})t_j^2)^{1/2}$.
Then for $1\leq j\leq k$, $u^2\geq C_7^22^{i_j-2}$ therefore $t_j\geq C_7\log(C_7^2n/u^2)$ for all $j$ and we may apply
Theorem \ref{thm:jointos} and get
\begin{align*}
\Pr(Y_1\geq u_1,\dots,Y_s\geq u_s)&=\Pr(X_{l_1}^*\geq t_1 ,\ldots,X_{l_k}\geq t_k)\leq
\exp\Big(-\sigma_X^{-1}\Big(\frac{1}{C_7}u\Big)\Big)
\\
&\leq \exp\Bigg(-\sigma_X^{-1}\Bigg(\frac{1}{2C_7}\sqrt{\sum_{j=1}^s 2^{j}u_j^2}\Bigg)\Bigg). 
\end{align*}
\end{proof}

\begin{lem}
\label{lem:prsum}
For nonnegative r.v.'s $Y_1,\ldots,Y_s$ and $u>0$ we have
\[
\Pr\Big(\sum_{i=1}^s Y_i\geq u\Big)\leq 
\sum_{(k_1,\ldots,k_s)\in I_s}\Pr\Big(Y_1\geq \frac{k_1u}{2s},\ldots,Y_n\geq \frac{k_su}{2s} \Big),
\]
where
\[
I_s:=\{k_1,\ldots,k_s\in \{0,1,\ldots,s\}^s\colon\ k_1+\ldots+k_s=s\}.
\]
\end{lem}

\begin{proof}
It is enough to observe that if $y_1+\ldots+y_s\geq u$ and we set $l_i:=\lfloor 2sy_i/u\rfloor$ then $y_i\geq l_iu/(2s)$ and
$\sum_{i=1}^sl_i\geq \sum_{i=1}^s (2sy_i/u-1)\geq s$.
\end{proof}

\begin{proof}[Proof of Theorem \ref{thm:lrPaour}]
Let $s:=1+\lfloor \log_2 n\rfloor$ and $Y_j$, $1\leq j\leq s$ be as in Corollary \ref{cor:jointos}.
We have
\begin{align*}
\|X\|_r^r
&=\sum_{i=1}^n |X_i^*|^r\leq \sum_{j=1}^s2^{j-1}|X_{2^{j-1}}^*|^r\leq
\sum_{j=1}^s2^{r+j-1}(Y_j^r+C_7^r\log^r(4n2^{-j}))
\\
&\leq (C_8r)^rn+\sum_{j=1}^s2^{r+j-1}Y_j^r.
\end{align*}
By Lemma \ref{lem:prsum}
\[
\Pr\Big(\sum_{j=1}^s2^{r+j-1}Y_j^r\geq u^r\Big)\leq
\sum_{(k_1,\ldots,k_s)\in I_s}\Pr\Big(2Y_1^r\geq \frac{k_1u^r}{s2^r},\ldots,2^{s}Y_s^2\geq \frac{k_su^r}{s2^r}\Big).
\]
Moreover for any $(k_1,\ldots,k_s)\in I_s$,
\[
\sum_{j=1}^s2^{j-2j/r}\Big(\frac{k_j}{s}\Big)^{2/r}\geq 
\sum_{j=1}^s\Big(\frac{k_j}{s}\Big)^{2/r}\geq
\Big(\sum_{j=1}^s \frac{k_j}{s}\Big)^{2/r}=1.
\]
Therefore Corollary \ref{cor:jointos} yields
\[
\Pr\Big(\sum_{j=1}^s2^{r+j-1}Y_j^r\geq u^r\Big)\leq |I_s|\exp\Big(-\sigma_X^{-1}\Big(\frac{u}{4C_7}\Big)\Big).
\]
However $|I_s|=\binom{2s-1}{s-1}\leq 2^{2s-2}\leq n^2$, so we obtain for $u\geq 2C_8rn^{1/r}$,
\[
\Pr(\|X\|_r\geq u)\leq n^2\exp\Big(-\sigma_X^{-1}\Big(\frac{u}{8C_7}\Big)\Big).
\]
Since $rn^{1/r}\geq e\log n$ and for $\lambda,s\geq 1$, $\sigma_X^{-1}(2\lambda s)\geq \lambda \sigma_{X}^{-1}(s)$
and $\sigma_{X}^{-1}(s)\geq s$ we get
\[
\Pr(\|X\|_r\geq Ct)\leq \exp(-\sigma_{X}^{-1}(t))\quad \mbox{ for }t\geq rn^{1/r}. 
\]
Integration by parts easily yields \eqref{eq:lrPaour}.
\end{proof}

\begin{proof}[Proof of Theorem \ref{thm:unifP}]
Let $s:=1+\lfloor \log_2 m\rfloor$ and $Y_j$, $1\leq j\leq s$ be as in Corollary \ref{cor:jointos}.
Then
\begin{align*}
\sup_{|I|=m}|P_IX|^2
&=\sum_{i=1}^m|X_i^*|^2\leq \sum_{j=1}^{s}2^{j-1}|X_{2^{j-1}}^*|^2
\leq 
\sum_{j=1}^{s}2^{j}(C_7^2\log^2(4n2^{-j})+Y_j^2)
\\
&\leq C_9m\log^2(en/m)+\sum_{j=1}^s2^jY_j^2.
\end{align*}
Moreover,
\begin{align*}
\Pr\Big(\sum_{j=1}^s2^jY_j^2\geq u^2\Big)&\leq
\sum_{(k_1,\ldots,k_s)\in I_s}\Pr\Big(2Y_1^2\geq \frac{k_1u^2}{2s},\ldots,2^sY_s^2\geq \frac{k_su^2}{2s}\Big)
\\
&\leq |I_s|\exp\Big(-\sigma_X^{-1}\Big(\frac{u}{2\sqrt{2}C_7}\Big)\Big),
\end{align*}
where the first inequality follows by Lemma \ref{lem:prsum} and the second one by Corollary \ref{cor:jointos}.
Observe that $|I_s|=\binom{2s-1}{s-1}\leq 2^{2s-2}\leq m^2$, thus we showed that for 
$u\geq \sqrt{2C_9m}\log(en/m)$
\[
\Pr\Big(\max_{|I|=m}|P_IX|\geq u\Big)\leq m^2\exp\Big(-\sigma_X^{-1}\Big(\frac{u}{4C_7}\Big)\Big).
\]
Since for $\lambda,s\geq 1$, $\sigma_X^{-1}(2\lambda s)\geq \lambda \sigma_{X}^{-1}(s)$  and $\sigma_{X}^{-1}(s)\geq s$
we easily get for $t\geq 1$,
\[
\Pr\Big(\max_{|I|=m}|P_IX|\geq Ct\sqrt{m}\log(en/m)\Big)\leq 
\exp\Big(-\sigma_X^{-1}\Big(t\sqrt{m}\log(en/m)\Big)\Big).
\]
Theorem \ref{thm:unifP} follows by integration by parts.
\end{proof}

\noindent
Institute of Mathematics\\
University of Warsaw\\
Banacha 2\\
02-097 Warszawa\\
Poland\\
\texttt{rlatala@mimuw.edu.pl}
\end{document}